\documentclass{article}
\usepackage{amsmath,amsthm,amssymb,amsfonts}
\usepackage{verbatim}
\usepackage{graphicx}
\usepackage{stmaryrd} 
\usepackage{hyperref}

\usepackage[colorinlistoftodos]{todonotes}

\usepackage{tikz}
\usepackage{tkz-berge, tkz-graph}
\tikzset{vertex/.style={circle,draw,fill,inner sep=0pt,minimum size=1mm}}

\theoremstyle{plain}
\newtheorem{thm}{Theorem}
\newtheorem{lem}[thm]{Lemma}
\newtheorem{prop}[thm]{Proposition}
\newtheorem{cor}[thm]{Corollary}

\newtheorem{question}[thm]{Question}

\theoremstyle{definition}
\newtheorem{definition}[thm]{Definition}

\newtheorem{exl}[thm]{Example}

\numberwithin{thm}{section}

\newcommand{\adj}{\leftrightarrow}
\newcommand{\adjeq}{\leftrightarroweq}

\def\Z{{\mathbb Z}}

\def\N{{\mathbb N}}
\begin{document}

\title{Approximate Fixed Point Properties in Digital Topology}
\author{Laurence Boxer
\thanks{
    Department of Computer and Information Sciences,
    Niagara University,
    Niagara University, NY 14109, USA;
    and Department of Computer Science and Engineering,
    State University of New York at Buffalo.
    email: boxer@niagara.edu 
}
}

\date{}
\maketitle{}
\begin{abstract}
We study the approximate fixed point property (AFPP) for
continuous single-valued functions and for
continuous multivalued functions in digital topology. We extend what is known about
these notions and discuss errors that have appeared
in the literature.
\end{abstract}

\section{Introduction}
In digital topology, we study properties of digital
images that are inspired by classical topology.
A digital version of continuous functions has been
developed, and researchers have had success in
studying digital versions of connectedness, homotopy,
fundamental groups, homology, et al., such that
digital images resemble the Euclidean objects they model
with respect to these properties. 

However, the fixed point properties of digital images
are often quite different from those of their
Euclidean inspirations. E.g., while there are
many examples of topological spaces with fixed point 
property (FPP), it is known~\cite{BEKLL} that a 
digital image $X$ has the FPP if and only if $X$ has
a single point. Therefore, the study of
{\em almost}~\cite{Rosenfeld,Tsaur} or
{\em approximate}~\cite{BEKLL} {\em fixed points} and
the {\em almost/approximate fixed point property} (AFPP) is
often of interest.

The paper~\cite{CinKar} introduced the AFPP for
continuous multivalued functions on digital images and
obtained some results for this property, but provided
no examples of digital images with this property.
We provide examples in this paper
along with additional general results concerning the
AFPP for continuous single-valued and multivalued functions on 
digital images. We also discuss errors that appeared in
the papers~\cite{Han19,CinKar}.

\section{Preliminaries}
Much of this section is quoted or 
paraphrased from the references.

We use $\Z$ to indicate the set of integers.

\subsection{Adjacencies}
The $c_u$-adjacencies are commonly used.
Let $x,y \in \Z^n$, $x \neq y$, where we consider these points as $n$-tuples of integers:
\[ x=(x_1,\ldots, x_n),~~~y=(y_1,\ldots,y_n).
\]
Let $u \in \Z$,
$1 \leq u \leq n$. We say $x$ and $y$ are 
{\em $c_u$-adjacent} if
\begin{itemize}
\item there are at most $u$ indices $i$ for which 
      $|x_i - y_i| = 1$, and
\item for all indices $j$ such that $|x_j - y_j| \neq 1$ we
      have $x_j=y_j$.
\end{itemize}
Often, a $c_u$-adjacency is denoted by the number of points
adjacent to a given point in $\Z^n$ using this adjacency.
E.g.,
\begin{itemize}
\item In $\Z^1$, $c_1$-adjacency is 2-adjacency.
\item In $\Z^2$, $c_1$-adjacency is 4-adjacency and
      $c_2$-adjacency is 8-adjacency.
\item In $\Z^3$, $c_1$-adjacency is 6-adjacency,
      $c_2$-adjacency is 18-adjacency, and $c_3$-adjacency
      is 26-adjacency.
\item In $\Z^n$, $c_1$-adjacency is $2n$-adjacency and $c_n$-adjacency is $(3^n - 1)$-adjacency.
\end{itemize}

For $\kappa$-adjacent $x,y$, we write $x \adj_{\kappa} y$ or $x \adj y$ when $\kappa$ is understood.
We write $x \adjeq_{\kappa} y$ or $x \adjeq y$ to mean that either $x \adj_{\kappa} y$ or $x = y$.
We say subsets $A,B$ of a digital image $X$ are ($\kappa$-)adjacent, $A \adjeq_{\kappa} B$ or
$A \adjeq B$ when $\kappa$ is understood, if there exist $a \in A$ and $b \in B$ such that
$a \adjeq_{\kappa} b$.

We say $\{x_n\}_{n=0}^k \subset (X,\kappa)$ is a {\em $\kappa$-path} (or a {\em path} if $\kappa$ is understood)
from $x_0$ to $x_k$ if $x_i \adjeq_{\kappa} x_{i+1}$ for $i \in \{0,\ldots,k-1\}$, and $k$ is the {\em length} of the path.

Another adjacency we will use is the following.

\begin{definition}
{\rm \cite{StaeckerBU}}
In a digital image $(X,\lambda)$, $x \adj_{\lambda^k} y$ if there is a $\lambda$-path of length
at most~$k$ in $X$ from $x$ to $y$.
\end{definition}

A subset $Y$ of a digital image $(X,\kappa)$ is
{\em $\kappa$-connected}~\cite{Rosenfeld},
or {\em connected} when $\kappa$
is understood, if for every pair of points $a,b \in Y$ there
exists a $\kappa$-path in $Y$ from $a$ to $b$.

\subsection{Digitally continuous functions}
The following generalizes a definition of~\cite{Rosenfeld}.

\begin{definition}
\label{continuous}
{\rm ~\cite{Boxer99}}
Let $(X,\kappa)$ and $(Y,\lambda)$ be digital images. A single-valued function
$f: X \rightarrow Y$ is {\em $(\kappa,\lambda)$-continuous} if for
every $\kappa$-connected $A \subset X$ we have that
$f(A)$ is a $\lambda$-connected subset of $Y$.
If $(X,\kappa)=(Y,\lambda)$, we say such a function is {\em $\kappa$-continuous},
denoted $f \in C(X,\kappa)$.
$\Box$
\end{definition}

When the adjacency relations are understood, we will simply say that $f$ is \emph{continuous}. Continuity can be expressed in terms of adjacency of points:
\begin{thm}
{\rm ~\cite{Rosenfeld,Boxer99}}
A single-valued function $f:X\to Y$ is continuous if and only if $x \adj x'$ in $X$ implies $f(x) \adjeq f(x')$. \qed
\end{thm}

Composition preserves continuity, in the sense of the following.

\begin{thm}
{\rm \cite{Boxer99}}
\label{composition}
Let $(X,\kappa)$, $(Y,\lambda)$, and $(Z,\mu)$ be digital images.
Let $f: X \to Y$ be $(\kappa,\lambda)$-continuous and let
$g: Y \to Z$ be $(\lambda,\mu)$-continuous. Then
$g \circ f: X \to Z$ is $(\kappa,\mu)$-continuous.
\end{thm}

Given $X = \Pi_{i=1}^v X_i$, we denote throughout this paper the projection
onto the $i^{th}$ factor by $p_i$; i.e., $p_i: X \to X_i$ is defined by
$p_i(x_1,\ldots,x_v) = x_i$, where $x_j \in X_j$.

\subsection{Digitally continuous multivalued functions}
A \emph{multivalued function} $f$ from $X$ to $Y$ assigns a subset of $Y$ to each point of $x$. We will  write $f:X \multimap Y$. For $A \subset X$ and a multivalued function $f:X\multimap Y$, let $f(A) = \bigcup_{x \in A} f(x)$. 

The papers~\cite{egs08, egs12} define continuity for multivalued functions between digital images based on subdivisions. (These papers make an error with respect to compositions, that is corrected in \cite{gs15}.) We have the following.
\begin{definition}
\rm{\cite{egs08, egs12}}
For any positive integer $r$, the \emph{$r$-th subdivision} of $\Z^n$ is
\[ \Z_r^n = \{ (z_1/r, \dots, z_n/r) \mid z_i \in \Z \}. \]
An adjacency relation $\kappa$ on $\Z^n$ naturally induces an adjacency relation (which we also call $\kappa$) on $\Z_r^n$ as follows: $(z_1/r, \dots, z_n/r) \adj_{\kappa} (z'_1/r, \dots, z'_n/r)$ in $\Z^n_r$ if and only if
$(z_1, \dots, z_n) \adj_{\kappa}(z_1', \dots, z_n')$ in $\Z^n$.

Given a digital image $(X,\kappa) \subset (\Z^n,\kappa)$, the \emph{$r$-th subdivision} of $X$ is 
\[ S(X,r) = \{ (x_1,\dots, x_n) \in \Z^n_r \mid (\lfloor x_1 \rfloor, \dots, \lfloor x_n \rfloor) \in X \}. \]

Let $E_r:S(X,r) \to X$ be the natural map sending $(x_1,\dots,x_n) \in S(X,r)$ to $(\lfloor x_1 \rfloor, \dots, \lfloor x_n \rfloor)$. 

For a digital image $(X,\kappa) \subset (\Z^n,\kappa)$, a function $f:S(X,r) \to Y$ \emph{induces a multivalued function $F:X\multimap Y$} as follows:
\[ F(x) = \bigcup_{x' \in E^{-1}_r(x)} \{f(x')\}. \]

A multivalued function $F:X\multimap Y$ is called $(\kappa,\lambda)$-\emph{continuous} when there is some $r$ such that $F$ is induced by some single-valued 
$(\kappa,\lambda)$-continuous function $f:S(X,r) \to Y$. $\Box$
\end{definition}

\begin{figure}[htb]
\begin{center}
\includegraphics[width=4.0in]{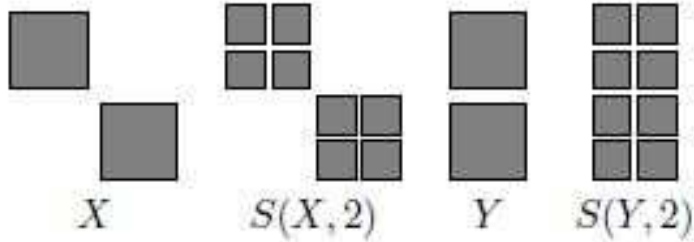}
\caption{Two images $X$ and $Y$ with their second subdivisions~\cite{BxSt16a}}
\label{subdivfig} 
\end{center}
\end{figure}


\begin{exl}
~\cite{BxSt16a}
An example of two digital images and their subdivisions is given in Figure \ref{subdivfig}.
Note that the subdivision construction (and thus the notion of continuity) depends on the particular embedding of $X$ as a subset of $\Z^n$. In particular we may have $X, Y \subset \Z^n$ with $X$ isomorphic to $Y$ but $S(X,r)$ not isomorphic to $S(Y,r)$. This is the case for the two images in Figure~\ref{subdivfig}, when we use 8-adjacency for all images: $X$ and $Y$ in the figure are isomorphic, each being a set of two adjacent points, but $S(X,2)$ and $S(Y,2)$ are not isomorphic since $S(X,2)$ can be disconnected by removing a single point, while this is impossible in $S(Y,2)$. $\Box$
\end{exl}

It is known~\cite{gs15} that a composition of 
digitally continuous multivalued functions need not
be continuous. However, we have the following.

\begin{thm}
{\rm \cite{gs15}}
\label{cv-composition}
Let $X \subset \Z^v$, $Y \subset \Z^n$, $Z \subset \Z^p$. Let
$F: X \multimap Y$ be a $(c_v,\kappa)$-continuous multivalued function and let
$G: Y \multimap Z$ be a $(\kappa, \lambda)$-continuous multivalued function.
Then $G \circ F: X \multimap Z$ is $(c_v, \lambda)$-continuous.
\end{thm}

Another way in which composition preserves continuity is the following.

\begin{prop}
\label{singleOfMultiple}
Let $F: X \multimap Y$ be a $(\kappa,\lambda)$-continuous multivalued function
and let $g: Y \to Z$ be a $(\lambda, \mu)$-continuous single-valued function.
Then $g \circ F: X \multimap Z$ is a $(\kappa, \mu)$-continuous multivalued function.
\end{prop}

\begin{proof}
Let $f: S(X,r) \to Y$ be a $(\kappa,\lambda)$-continuous function that
induces $F$. Then $g \circ f: S(X,r) \to Z$ is a $(\kappa, \mu)$-continuous
single-valued function that induces $g \circ F$.
\end{proof}

\subsection{Approximate fixed points}
\label{approxPrelim}
Let $f \in C(X,\kappa)$, let $F: X \multimap X$ be a $\kappa$-continuous multivalued function,
and let $x \in X$. We say
\begin{itemize}
    \item $x$ is a {\em fixed point} of $f$ if $f(x)=x$; 
          $x$ is a {\em fixed point} of $F$ if $x \in F(x)$.
    \item If $f(x) \adjeq_{\kappa} x$, then
          $x$ is an {\em almost fixed point}~\cite{Rosenfeld,Tsaur} or
          {\em approximate fixed point}~\cite{BEKLL} of 
          $(f,\kappa)$.
    \item If there exists $x' \in F(x)$ such that $x \adjeq_{\kappa} x'$, then $x$
          is an {\em approximate fixed point}~\cite{CinKar} of $(F,\kappa)$.
    \item A digital image $(X,\kappa)$ has the
          {\em approximate fixed point property with respect to continuous single-valued
          functions ($AFPP_S$)}~\cite{BEKLL} if for every $g \in C(X,\kappa)$
          there is an approximate fixed point of $g$.
    \item A digital image $(X,\kappa)$ has the
          {\em approximate fixed point property with respect to continuous multivalued
          functions ($AFPP_M$)}~\cite{CinKar} if for every 
          $(\kappa,\kappa)$-continuous multivalued
          function $G: X \multimap X$ there is an approximate fixed point of $G$.
\end{itemize}

\begin{thm}
{\rm \cite{BEKLL}}
\label{BEKLLisoThm}
Let $X$ and $Y$ be digital images such that $(X,\kappa)$ and
$(Y,\lambda)$ are isomorphic. If $(X,\kappa)$ has the $AFPP_S$,
then $(Y,\lambda)$ has the  $AFPP_S$.
\end{thm}

\begin{thm}
{\rm \cite{BEKLL}}
\label{BEKLLretractionThm}
Let $X$ and $Y$ be digital images such that 
$Y$ is a $\kappa$-retract of $X$. If $(X,\kappa)$ has the $AFPP_S$, then $(Y,\kappa)$ has the $AFPP_S$.
\end{thm}

The latter inspired the following.

\begin{thm}
{\rm \cite{CinKar}}
\label{CinKarRetractionThm}
Let $X \subset \Z^v$ and let $Y$ be a $c_v$-multivalued retract of $X$. 
If $(X,c_v)$ has the $AFPP_M$, then
$(Y,c_v)$ has the $AFPP_M$.
\end{thm}

\begin{prop}
\label{multiImpliesSingle}
If $(X,\kappa)$ has the $AFPP_M$, then $(X,\kappa)$
has the $AFPP_S$.
\end{prop}

\begin{proof}
The assertion follows from the observation that a continuous 
single-valued function between digital images is a 
continuous multivalued function.
\end{proof}

\section{Results for digital cubes}
In this section, we consider approximate fixed point
properties for digital cubes.
We start with the special case of dimension~1.

Theorem~3.3 of~\cite{Rosenfeld}, which proves that a digital
interval has the $AFPP_S$, is extended as follows.

\begin{thm}
\label{intervalAFPP}
The digital image $([a,b]_{\Z}, c_1)$ has the $AFPP_M$.
\end{thm}

\begin{proof}
We modify the proof given for Theorem~3.3
of~\cite{Rosenfeld}. 
Let $F: [a,b]_{\Z} \multimap [a,b]_{\Z}$ be
$c_1$-continuous.
If $a \in F(a)$ or $b \in F(b)$, we are done. 
Otherwise, 
let $f: (S(X,r),c_1) \to ([a,b]_{\Z},c_1)$ generate
$F$. Then $f(a)>a$ and $f(b)<b$,  so 
\begin{equation}
\label{g-vals}
    g(t) = f(t)-\lceil t \rceil \mbox{ is positive at $t=a$ and negative at $t=b$.}
\end{equation}
 
Since $f$ is continuous,
we must have $f(t - 1/r) \in \{f(t)-1, f(t), f(t)+1\}$ for all
$t \in S(X,r) \setminus \{a\}$. Let $z \in \Z$.
Since $\lceil t \rceil$ is constant for $z \le t < z+1$, 
and changes by 1 as $t$ increases from $z+(r-1)/r$ to $z+1$,
it follows that for $a \le t <b$, 
an increase of $1/r$ in the value of $t$ causes 
the expression $|g(t)|$
to change by at most 1 for $z \le t < z+1$,
and by at most $2$ for $t= z+(r-1)/r$.
It follows from~(\ref{g-vals}) that there exists $c$ such that
$g(c) \in \{-1,0,1\}$, i.e., $c$ is an approximate fixed point for $F$.
\end{proof}

A. Rosenfeld's paper~\cite{Rosenfeld} states the following as its Theorem~4.1 (quoted verbatim).
\begin{quote}
    Let $I$ be a digital picture, and let $f$ be a continuous function from $I$
    into $I$; then there exists a point $P \in I$ such that $f(P)=P$ or is a neighbor
    or diagonal neighbor of $P$.
\end{quote}
Several subsequent papers have incorrectly
concluded that this result implies that $I$ with
some $c_u$ adjacency has the $AFPP_S$. 
By {\em digital picture} Rosenfeld means a digital cube, $I= [0,n]_{\Z}^v$.
By a ``continuous function" he means a $(c_1,c_1)$-continuous function;
by ``a neighbor or diagonal neighbor of $P$" he means a $c_v$-adjacent point. Thus,
if we generalize our definition of the $AFPP_S$ as 
Definition~\ref{AFPP-general} below, what
Rosenfeld's theorem shows in terms of an AFPP is stated as Theorem~\ref{Rosenfeld'sAFPP} below.

\begin{definition}
\label{AFPP-general} Let $\kappa,\lambda, \mu$
be adjacencies for a digital image $X$.
Then $X$ has the {\em approximate fixed point property for
single-valued functions and}
$(\kappa,\lambda,\mu)$, denoted $AFPP_S(\kappa,\lambda,\mu)$, 
if for every $(\kappa,\lambda)$-continuous single-valued function
$f: X \to X$ there exists $x \in X$ such that $x \adjeq_{\mu} f(x)$.
$X$ has the {\em approximate fixed point property for
multivalued functions and}
$(\kappa,\lambda,\mu)$, denoted $AFPP_M(\kappa,\lambda,\mu)$, 
if for every $(\kappa,\lambda)$-continuous multivalued function
$F: X \multimap X$ there exist $x \in X$ and $y \in F(x)$
such that $x \adjeq_{\mu} y$.
\end{definition}

Thus, $(X,\kappa)$ has the $AFPP_S$ if and only if $X$ has the $AFPP_S(\kappa,\kappa,\kappa)$. 

\begin{thm}
\label{Rosenfeld'sAFPP}
{\rm \cite{Rosenfeld}}
Let $X = [0,n]_{\Z}^v \subset \Z^v$
for some $v \in \N$. Then $X$ has the
$AFPP_S(c_1,c_1,c_v)$.
\end{thm}

Since Rosenfeld's proof can be easily modified to any digital
cube $\Pi_{i=1}^v [a_i,b_i]_{\Z}$, and since
for $1 \le u \le v$, a $(c_u,c_1)$-continuous $f: X \to X$ is
$(c_1,c_1)$-continuous~\cite{BxAlt}, we have the following.

\begin{cor}
\label{cube} Let $n$ and $v$ be positive integers.
Let $X = \Pi_{i=1}^v [a_i,b_i]_{\Z} \subset \Z^v$.
Let $u \in [1,v]_{\Z}$.
Then $X$ has the $AFPP_S(c_u,c_1,c_v)$.
\end{cor}

\begin{thm}
\label{digitalPicProp}
Let $X = [0,n]_{\Z}^v \subset \Z^v$.
Then $X$ has the $AFPP_M(c_v,c_v,c_v^{\lceil n/2 \rceil})$. 
\end{thm}

\begin{proof}
Let $x$ be a point of $X$ such that each coordinate of $x$ is a member of
$\{\lfloor n/2 \rfloor, \lceil n/2 \rceil\}$. Let $F: X \multimap X$ be
$(c_v,c_v)$-continuous. Then for some (indeed, every) $y \in F(x)$, there is a $c_v$-path
from $x$ to $y$ of length at most $\lceil n/2 \rceil$.
\end{proof}

As in Proposition~\ref{multiImpliesSingle}, we have the following.

\begin{cor}
Let $X = [0,n]_{\Z}^v \subset \Z^v$.
Then $X$ has the $AFPP_S(c_v,c_v,c_v^{\lceil n/2 \rceil})$. 
\end{cor}

The following (restated here in our terminology) is Theorem~1 of Han's paper~\cite{Han19}.
\begin{thm}
\label{HanThm}
Let $X = [-1,1]_{\Z}^v$ and $1 \le u \le v$. Then
$(X,c_u)$ has the $AFPP_S$ if and only if $u=v$.
\end{thm}

We show below that Theorem~\ref{HanThm} is correct. This is necessary, since
Han fails to give a correct proof for either 
implication of this theorem.
\begin{itemize}
    \item Han offers two ``proofs" of the assertion that $(X,c_v)$ has 
          the $AFPP_S$. Both of Han's arguments are incorrect.
          \begin{enumerate}
              \item Han's first ``proof" of the assertion that $(X,c_v)$ has 
          the $AFPP_S$ is based on the false assertion that $f \in C(X,c_u)$ implies 
          $f \in C(X,c_v)$. The latter assertion is
          incorrect, as shown in Example~\ref{refuteHan} below.
          \item Han's second ``proof" argues that
                assuming otherwise yields a
                contradiction. He gives an example, for
                $v=2$, of a self-map on $X$. He
                claims without explanation that
                if $(X,c_v)$ fails to have 
                the $AFPP_S$, then this map shows
                that all self maps on $(X,c_v)$ are discontinuous; further, he offers
                no argument that this example generalizes to all self maps on $X$ for
                all dimensions~$v$. 
          \end{enumerate} 
          That $(X,c_v)$ has the $AFPP_S$ is proven 
          below at Example~\ref{AFPPexl}(3).
          Han's focus on the point $(0,0)$ raises
          the possibility that he had such an example
          in mind for his second ``proof," but
          this is not clear.
    \item Han's argument for the assertion that $(X,c_u)$ 
          does not have the $AFPP_S$ for $u < v$ is incorrect, as 
          it discusses only particular self-maps
          on the digital images $([-1,1]_{\Z}^2, c_1)$ and
          $([-1,1]_{\Z}^3, c_2)$,
          with no indication that these examples
          generalize. Theorem~\ref{containsUCube}, below, correctly shows this assertion.
\end{itemize}

\begin{exl}
\label{refuteHan}
Let $X=[-1,1]_{\Z}^2$. Let $f: X \to X$ be defined by
\[ f(x,y) = \left \{ \begin{array}{ll}
     (\min\{1,x+1\}, y) &  \mbox{if } y \in \{-1,0\}; \\
     (x,0) & \mbox{if } y=1.
    \end{array}    \right .
\]
See 
Figure~2.
Then it is easily seen that $f \in C(X,c_1)$. However, $f \not \in C(X,c_2)$, since
$(-1,1) \adj_{c_2} (0,0)$, but $f(-1,1)=(-1,0)$ and $f(0,0) = (1,0)$ are not $c_2$-adjacent.
\end{exl}

\begin{figure}[htb]
\begin{center}
\includegraphics[width=3in]{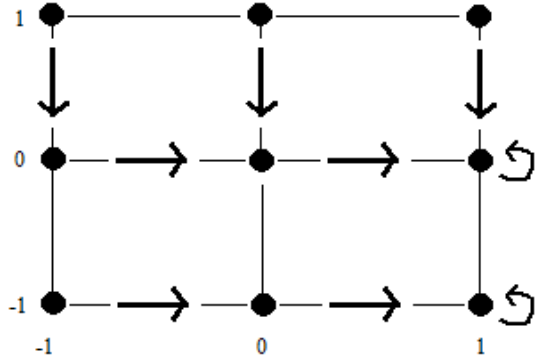}
\label{refuteHanIllustration}
\caption{The function $f$ of Example~\ref{refuteHan}. For each $x \in X$, an arrow
indicates $f(x)$. This function belongs to $C(X,c_1)$, but not to $C(X,c_2)$.}
\end{center}
\end{figure}

\begin{prop}
\label{uCube}
Let $v$ be a positive integer, $v > 1$.
Let $X = [0,1]_{\Z}^v \subset \Z^v$.
Let $u \in [1,v-1]_{\Z}$.
Then $(X,c_u)$ does not have the $AFPP_S$.
\end{prop}

\begin{proof}
Let $f: [0,1]_{\Z} \to [0,1]_{\Z}$ be defined by $f(z)=1-z$. Let $F: X \to X$ be
defined by
\[ F(x_1,\ldots,x_v) = (f(x_1), \ldots, f(x_v)).
\]
It is easily seen that if $x, x' \in X$, then $F(x)$ and $F(x')$ differ in as many
coordinates as $x$ and $x'$, and therefore $F \in C(X,c_u)$. However, for
each $x \in X$, $x$ and $F(x)$ differ in all $v$ coordinates, so $F$ has no
$c_u$-approximate fixed point.
\end{proof}

\begin{thm}
\label{containsUCube}
Let $X \subset \Z^v$ be such that $X$ has a subset $Y = \Pi_{i=1}^v [a_i,b_i]_{\Z}$, where
$v>1$; for all indices $i$, $b_i \in \{a_i,a_i + 1\}$; and, for at least 2 indices $i$,
$b_i=a_i+1$. Then $(X,c_u)$ fails to have the $AFPP_S$ for $1 \le u < v$. (This contains one
of the assertions of Theorem~\ref{HanThm}.)
\end{thm}

\begin{proof}
Let $r: X \to Y$ be defined by its coordinate functions,
\[ p_i(r(x)) = \left \{ \begin{array}{ll}
   a_i & \mbox{if } p_i(x) \le a_i; \\
   b_i & \mbox{if } p_i(x) \ge b_i.
   \end{array} \right .
\]
It is easily seen that $r$ is a $c_u$-retraction of $X$ to $Y$. The assertion 
follows from Proposition~\ref{uCube} and Theorem~\ref{BEKLLretractionThm}.
\end{proof}

We expand on one of the assertions of Theorem~\ref{HanThm} as follows.

\begin{cor}
Let $X = \Pi_{i=1}^v [a_i,b_i]_{\Z}$ such that $b_i > a_i$ for at least
2 indices~$i$. Let $u \in \Z$, $1 \le u < v$.
Then $(X,c_u)$ does not have the $AFPP_S$.
\end{cor}

\begin{proof}
It follows easily from
Theorem~\ref{containsUCube} that $(X,c_u)$ does not have the $AFPP_S$.
\end{proof}

\begin{lem}
\label{AFPPforUnitCubes}
Let $X = [0,1]_{\Z}^v \subset \Z^v$. Then
given $u \in [1,v]_{\Z}$ and a $(c_u,c_v)$-continuous multivalued function
$F: X \multimap X$, every $x \in X$ satisfies 
$\{x\} \adjeq_{c_v} F(x)$.
\end{lem}

\begin{proof}
The assertion follows from the observation that $(X,c_v)$ is a complete graph.
\end{proof}

Theorem~\ref{containsUCube} states a severe limitation on the $AFPP_S$ and
on the $AFPP_M$ for digital images $X \subset \Z^v$ and the $c_u$ adjacency,
where $1 \le u < v$. We now consider the case $u=v$.
 
The paper~\cite{CinKar} introduces the $AFPP_M$ but provides
no nontrivial examples of digital images with this property.
Simple examples are given in the following.

\begin{exl}
\label{AFPPexl}
The following digital images have the $AFPP_M$.
\begin{enumerate}
    \item A singleton.
    \item $([0,1]_{\Z}^v, c_v)$.
    \item $([-1,1]_{\Z}^v, c_v)$. (By Proposition~\ref{multiImpliesSingle},
          this contains one of the assertions of Theorem~\ref{HanThm}.)
\end{enumerate}
\end{exl}

\begin{proof}
\begin{enumerate}
    \item The case of a singleton is trivial.
    \item The assertion for $([0,1]_{\Z}^v, c_v)$ follows
          from Lemma~\ref{AFPPforUnitCubes}.
    \item The assertion for $([-1,1]_{\Z}^v, c_v)$ follows
          from the observation that the point $(0,0,\ldots,0)$ is
          $c_v$-adjacent to every other point of the image,
          hence must be an approximate fixed point for every
          $(c_v,c_v)$-continuous multivalued self-map on this image.
\end{enumerate}
\end{proof}

\section{Retraction and preservation of $AFPP_M$}
\label{retractSec}
Retraction preserves the $AFPP_S$~\cite{BEKLL}.
The paper~\cite{CinKar} claims the following analog for the $AFPP_M$
(restated here in our terminology) as its Theorem~4.4.
\begin{quote}
Let $X \subset \Z^v$ such that $(X,c_v)$ has the $AFPP_M$. Let $Y \subset X$ be a $(c_v,c_v)$-continuous
multivalued retract of $X$. Then $(Y,c_v)$ has the $AFPP_M$.
\end{quote}
However, there are errors in the argument offered as proof of this claim, so the
assertion must be regarded as unproven. The authors argue as follows.
Given a $(c_v,c_v)$-continuous multivalued $F: Y \multimap Y$, 
let $I: Y \to X$ be the inclusion and $R: X \multimap Y$ a $(c_v,c_v)$-continuous
multivalued retraction. Then $G = I \circ F \circ R: X \multimap X$ is shown to be
$(c_v,c_v)$-continuous and therefore has an approximate fixed point $x_0$.
Thus, there exists $x_1 \in G(x_0)$ such that $x_1 \adjeq_{c_v} x_0$. By continuity of $G$ we have
\[ x_1 \in G(x_0) \adjeq_{c_v} G(x_1) 
   = I \circ F \circ R(x_1).
\]
\begin{itemize}
\item There follows the claim that the latter is equal to  $I \circ F(x_1)$; this is
unjustified since we do not know whether $x_1$ belongs to $Y$. Indeed, we do not know if
$F(x_1)$ is defined.
\item Further, after observing that we would have $I \circ F(x_1) = F(x_1)$, it is claimed that
      $x_1 \adjeq_{c_v} F(x_1)$, but this is unjustified since we can not assume
      that $G(x_0)$ and $G(x_1)$ are singletons.
\end{itemize}

\section{Universal and weakly universal multivalued functions}
Universal and weakly universal single-valued functions for
digital images were introduced in~\cite{BEKLL} 
and~\cite{BoxerRFP3}, respectively. The paper~\cite{CinKar}
seeks to obtain analogous results for multivalued functions.
In this section, we correct a small error of~\cite{CinKar} 
in its treatment of universal multivalued functions. 
The error of concern parallels an error of~\cite{BEKLL},
and our corrections parallel those in~\cite{BoxerRFP3}.
The error is due to using {\em universal} multivalued 
functions rather than {\em weak universal} multivalued
functions (see Definition~\ref{univM}, below);
this error propagates through multiple assertions of~\cite{CinKar}.

\begin{definition}
\label{univM}
Let $(X,\kappa)$ and $(Y, \lambda)$ be digital images. Let $F: X \multimap Y$ be a
$(\kappa, \lambda)$-continuous multivalued function.
\begin{itemize}
    \item $F$ is a {\em universal} for $(X, Y)$~{\rm \cite{CinKar}} if given a $(\kappa, \lambda)$-continuous
          multivalued function $G: X \multimap Y$, there exist $x \in X$ and $y \in F(x)$,
          $y' \in G(x)$ such that $y \adj_{\lambda} y'$.
    \item $F$ is a {\em weak universal} for $(X, Y)$ if given a $(\kappa, \lambda)$-continuous
          multivalued function $G: X \multimap Y$, there exist $x \in X$ and $y \in F(x)$,
          $y' \in G(x)$ such that $y \adjeq_{\lambda} y'$.
\end{itemize}
\end{definition}

Proposition~3.1 of~\cite{CinKar} asserts the following (quoted verbatim).
\begin{quote}
    Let $X$ and $Y$ be digital images. Suppose $Y$ is finite. Then
    the multivalued function $F: X \multimap Y$ defined by $F(x) = Y$ for all $x \in X$ is universal.
\end{quote}

This assertion is incorrect, as shown by the following. Let $(X,\kappa)=(Y,\kappa)$ be a digital
image with a single point $x_0$. Since there is no point in $X$ adjacent to $x_0$, no universal for $(X,X)$ exists, contrary
to the assertion of Proposition~3.1 of~\cite{CinKar}.
However, we have the following (note we do not need to assume that $Y$ 
is finite).

\begin{prop}
Let $X$ and $Y$ be digital images. 
Then the multivalued function $F: X \multimap Y$ defined by $F(x) = Y$ for 
all $x \in X$, is a weak universal.
\end{prop}

\begin{proof}
This follows from Definition~\ref{univM}.
\end{proof}

\begin{definition}
{\rm ~\cite{CinKar}}
A multivalued function $F: (X,\kappa) \multimap (Y,\lambda)$ is {\em injective}
if $F(x)=F(y)$ implies $x=y$. A injective multivalued function $G: X \multimap X$ is an
{\em identity} if $x \in G(x)$ for all $x \in X$.
\end{definition}
 
Proposition~4.1 of~\cite{CinKar} asserts the following (restated here in our terminology).
\begin{quote}
    Let $X$ be a digital image. Then $(X, \kappa)$ has the $AFPP_M$ if and only if
    an identity multivalued function is universal.
\end{quote}

This assertion is incorrect, as shown by the example of a digital image $X$ with a single point
$x_0$; $X$ trivially has the $AFPP_M$, but no multivalued
identity is universal since $x_0$ has no
adjacent point~$y$. However, we have the following (notice we show that we can take our
identity multivalued function to be the unique single-valued identity function $1_X$).

\begin{prop}
\label{M-andWeak}
 Let $X$ be a digital image. Then $(X, \kappa)$ has the $AFPP_M$ if and only if
the identity function $1_X$ is a weak universal.
\end{prop}

\begin{proof}
Our argument requires only minor changes in the 
argument given for its analog in~\cite{CinKar}.

Suppose $(X, \kappa)$ has the $AFPP_M$. Then given a $(\kappa,\kappa)$-continuous
multivalued function $F: X \multimap X$, there exist $x \in X$ and $y \in F(x)$ such
that $1_X(x)=x \adjeq y \in F(x)$. Therefore, $1_X$ is a weak universal.

Suppose $1_X$ is a weak universal. Then for any $(\kappa,\kappa)$-continuous
multivalued function $F: X \multimap X$, there exists $x \in X$ such that
$x = 1_X(x) \adjeq_{\kappa} y$ for some $y \in F(x)$.
Thus, $(X, \kappa)$ has the $AFPP_M$.
\end{proof}

\section{Further remarks}
We have studied approximate fixed point properties for
both single-valued and multivalued digitally continuous
functions on digital images. 

The question of whether the converse of 
Proposition~\ref{multiImpliesSingle} is valid appears to be
a difficult problem. I.e., we have the following.

\begin{question}
\label{singleImplyMultiple}
    If $(X,\kappa)$ has the $AFPP_S$, does 
    $(X,\kappa)$ has the $AFPP_M$?
\end{question}

We also have not answered the following.

\begin{question}
    Let $X = \Pi_{i=1}^v [a_i,b_i]_{\Z}$, where for at least
    2 indices~$i$ we have $b_i > a_i$. Does $(X,c_v)$ have the
    $AFPP_S$?
\end{question}

We saw in Section~\ref{retractSec} that the following question remains unanswered.
\begin{question}
    Let $X \subset \Z^v$ such that $(X,c_v)$ has the $AFPP_M$. Let $Y \subset X$ be a $(c_v,c_v)$-continuous
multivalued retract of $X$. Does $(Y,c_v)$ have the $AFPP_M$?
\end{question}

Since~\cite{BoxerRFP3} $1_X$ is a weak universal function for
single-valued continuous self-maps on $(X,\kappa)$ if and
only if $(X,\kappa)$ has the $AFPP_S$, in view of
Proposition~\ref{M-andWeak}, a positive solution to the
following question would yield a positive solution to
Question~\ref{singleImplyMultiple}.

\begin{question}
    If $1_X$ is a weak universal function for $C(X,\kappa)$,
    is $1_X$ is a weak universal for continuous multivalued
    functions on $(X,\kappa)$?
\end{question}

\section{Acknowledgement}
The suggestions and corrections of an anonymous reviewer are gratefully acknowledged.

\bibliographystyle{amsplain}

\end{document}